\documentclass{birkjour}

\usepackage{graphicx}
\usepackage{enumerate}
\usepackage{url}
\usepackage{hyperref}
\usepackage{mathtools}
\renewcommand{\O}{\mathcal{O}}

\newcommand{\rr}{\mathbb{R}}

\newcommand{\cc}{\mathbb{C}}

\newcommand{\nn}{\mathbb{N}}
\newcommand{\dd}{\mathbb{D}}
\newcommand{\tlr}{\mathcal{T}}
\newcommand{\overbar}[1]{\mkern 1.5mu\overline{\mkern-1.5mu#1\mkern-1.5mu}\mkern 1.5mu}

\newcommand{\setof}[1]{\left\{#1\right\}}

\newcommand{\abs}[1]{\left| #1 \right|}
\DeclarePairedDelimiter{\norm}{\lVert}{\rVert}
\newcommand{\paren}[1]{\left( #1 \right)}
\renewcommand{\L}{\ell^1}
\newcommand{\ck}{Cauchy-Kovalevskaya }
\newcommand{\pwrs}{\sum_{j = 0}^{\infty}} % unify power series indices
\newcommand{\tail}{_{\text{tail}}} % define tail subspace projection notation
\newcommand{\Frech}{Fr\'{e}chet }

\newtheorem{theorem}{Theorem}%[section]

\newtheorem{lemma}[theorem]{Lemma}
\newtheorem{proposition}[theorem]{Proposition}

\theoremstyle{definition}
\newtheorem{definition}{Definition}
\newtheorem{example}{Example}

\theoremstyle{remark}

\begin{document}
	
%\linenumbers

\title[]{A constructive proof of the  Cauchy-Kovalevskaya theorem for ordinary differential equations}

%\author{Shane Kepley         \and
%	Tianhao Zhang %etc.
%}
%

%\institute{Shane Kepley \at
%	Rutgers University, Department of Mathematics \\
%	\email{sk2011@math.rutgers.edu}           %  \\
%	%             \emph{Present address:} of F. Author  %  if needed
%	\and
%	Tianhao Zhang \at
%	Zhejiang University, School of Mathematical Sciences \\
%	\email{3160105282@zju.edu.cn}
%}

%%----------Author 1
\author[Kepley]{Shane Kepley*}
\address{Rutgers University, Department of Mathematics}
\email{sk2011@math.rutgers.edu}
\thanks{*Corresponding author}

\author[Zhang]{Tianhao Zhang}
\address{Zhejiang University, School of Mathematical Sciences}
	\email{3160105282@zju.edu.cn}
%%----------classification, keywords, date
%\subjclass{Primary 99Z99; Secondary 00A00}
%
%\keywords{Class file, journal}

\date{\today}

%
%%\thanks{This work was completed with the support of our
%%	\TeX-pert.}
%%----------Author 2
%\author[]{A Second Author}
%\address{The address of\br
%	the second author\br
%	sitting somewhere\br
%	in the world}
%\email{dont@know.who.knows}

\begin{abstract}
	We give a constructive proof of the classical Cauchy-Kovalevskaya theorem for ordinary differential equations which provides a sufficient condition for an initial value problem to have a unique, analytic solution. Our proof is inspired by a modern numerical technique for rigorously solving nonlinear problems known as the radii polynomial approach. The main idea is to recast the existence and uniqueness of analytic solutions as a fixed point problem on an appropriately chosen Banach space, and then prove a fixed point exists via a constructive version of the Banach fixed point theorem. A key aspect of this method is the use of an approximate solution which plays a crucial role in the theoretical proof. Our proof is constructive in the sense that we provide an explicit recipe for constructing the fixed point problem, an approximate solution, and the bounds necessary to prove the existence of the fixed point.   
\end{abstract}

\maketitle

\section{Introduction}
\label{sec:introduction}
In this paper we present a novel proof of the \ck theorem in the ordinary differential equation (ODE) setting. The general theorem, first proved by Sonya Kovalevskaya in 1874, gives sufficient conditions for a Cauchy problem to have a unique analytic solution. Unfortunately, spaces of analytic functions are typically not the right regularity for studying solutions of partial differential equations (PDE) so the \ck theorem is rarely practically applicable in this setting. On the other hand, the \ck theorem is often applicable to initial value problems (IVP) arising from ODEs which is the focus of the present work. We begin by stating the theorem in this setting. A statement of the general theorem and its classical proof can be found in most introductory PDE texts e.g.~\cite{Ebert2018a}. 
\begin{theorem}[\ck]
	\label{thm:ck_for_ODE}
	Suppose $V \subset \rr^n$ is an open subset and $f \colon V \to \rr^n$  is an analytic vector field. Then, the initial value problem
	\begin{equation}
	\dot x = f(x) \qquad x(0) = x_0 \in V
	\label{eq:IVP}
	\end{equation}
	has a unique solution which is analytic on some open interval, $J(x_0)$, containing zero. 
\end{theorem}
There are several proofs of this theorem in the literature. The classical proof provides a prototypical example of the {\em method of majorants}. In order to illustrate the constructive aspect of our approach, we sketch a version of the classical proof for the case $n = 1$. 

The main idea in the classical proof is to use the Taylor coefficients of $f$ to dominate the Taylor coefficients of $x$. Roughly speaking, $f$ is analytic if its Taylor coefficients decay rapidly enough. The classical proof follows from showing that this condition forces the Taylor coefficients of any solution to decay rapidly as well. Note that the existence and uniqueness of a solution on some open interval, $J(x_0)$, containing zero follows from the Picard-Lindel\"of theorem. In fact, by the usual bootstrap argument this theorem shows that this solution is as smooth as $f$. Hence, we may take for granted the existence of $x \in C^\infty(J(x_0))$ satisfying Equation \eqref{eq:IVP}. 

The \ck theorem asserts that in fact, $x \in C^\omega(J(x_0))$. Equivalently, there exists $\tau > 0$,  such that the series
\begin{equation}
x(t) = 	\sum_{j = 0}^\infty \frac{x^{(j)}(0)}{j!} t^j \qquad \abs{t} < \tau 
\end{equation}
converges. The crux of the classical argument arises from applying the Fa\`{a} di Bruno formula for the iterated chain rule with the assumption that $x$ satisfies Equation \eqref{eq:IVP}, to obtain the formula
\begin{equation}
\label{eq:majorizing_sequence}
x^{(j)}(0) = p_j \paren{f(0), f'(0), \dotsc, f^{(j-1)}(0)} \qquad j \in \nn,
\end{equation}
where each $p_j$ is a polynomial in $j$ variables with non-negative coefficients. Then, one defines the non-negative sequence, $\setof{u_j} = \setof{\abs{f^{(j)}(0)} : j \in \nn}$, so that we have the bound 
\begin{equation}
\label{eq:majorizing_bound}
\abs{x^{(j)}(0)} \leq p_j(u_0,\dotsc,u_{j-1}) \qquad \text{for all} \ j \in \nn.
\end{equation}
This bound implies that the function 
\begin{equation}
\label{eq:majorizing_function}
\tilde x(t) := \sum_{j = 0}^\infty \frac{p_j \paren{u_0, \dotsc, u_{j-1}} }{j!} t^j 
\end{equation}
is a majorant for $x$. The classical proof is concluded by showing that $\tilde x$ is analytic which ultimately follows as a consequence of the fact that $f$ is analytic. 

The classical proof is quite beautiful, however, we note that it is not constructive. In contrast with this approach, our proof of the \ck theorem is based on analyzing the coefficients of the solution and proving directly that they decay sufficiently fast. Our proof is inspired by the so called ``radii polynomial approach'', which provides a constructive framework for proving theorems in nonlinear analysis with assistance of a digital computer. While our proof does not have a numerical aspect, it is carried out in the same style so we briefly review the method.

\subsection{The radii polynomial approach}
The radii polynomial approach is a modern methodology combining functional analytic tools with rigorous numerical computations to study nonlinear problems.  The method first appeared in \cite{MR2338393} as a modification of the technique presented in \cite{MR1639986} for rigorously proving the existence of solutions of zero-finding problems using Newton's method. Since then, the radii polynomial approach has played an important role in a number of results in dynamical systems such as existence of spontaneous periodic solutions in the Navier-Stokes equations \cite{Berg2019}, chaos in the circular restricted four body problem \cite{Kepley2019}, coexistence of hexagonal patterns and rolls in the Swift-Hohenberg equations \cite{BDLM}, and the proof of Wright's conjecture \cite{VANDENBERG20187412}, to name just a few. This is a small subset of the growing collection of results which utilize the radii polynomial approach as the basis for rigorous numerical algorithms for computation and continuation of equilibria, periodic orbits, connecting orbits, solutions of initial/boundary value problems, and invariant manifolds (see e.g.\ \cite{MR2718657}, \cite{Berg2019a},  \cite{chebManifolds}, \cite{Lessard2014},  \cite{MR3068557},  \cite{manifoldPaper1},  \cite{parmChristian}, \cite{VanDenBerg2016}, \cite{Gonzalez2017}). A more detailed exposition on rigorous numerical techniques and various applications of radii polynomial approach can be found in \cite{jpjbReview}, \cite{VandenBerg2018}.

The main idea is to first recast problems as a zero-finding problem on a Banach space. Then, a Newton-like operator is introduced which has fixed points in one-to-one correspondence with solutions of the zero-finding problem. By combining careful ``pencil and paper'' estimates with rigorous computations, one tries to prove that this Newton-like operator has a fixed point by an application of the Banach fixed point theorem. If successful, the existence of a zero for the original problem is concluded.
 
In \cite{JayAMSChapter} the radii polynomial approach was generalized to a rigorous numerical IVP solver for polynomial vector fields in  which the fixed-point problem does not arise from a Newton-like operator. This approach was based on modifying the methodology in \cite{MR3281845} in which one looks for a fixed point of a ``Picard-like'' operator. In this work we follow a similar approach.  The main idea is to associate any instance of Equation \eqref{eq:IVP}, with a mapping, $T : X \to X$, where $X$ is an appropriate space of rapidly decaying real sequences. We provide an explicit construction for $X$ and $T$ depending only on $f$ and $x_0$, and we prove that if $T$ has a fixed point, then the solution of Equation \eqref{eq:IVP} is analytic. The \ck theorem follows after proving that if $f$ is analytic, then our construction always produces a map with a fixed point. 

We begin by describing the main theorem utilized in our approach which is a constructive version of the Banach fixed point theorem. 
\begin{theorem}
	\label{thm:local_radii_polynomial}
	Suppose that $X$ is a Banach space with norm $\norm{\cdot}_{X}$, $U\subset X$ is an open subset, and $T:U \to X$ is a \Frech differentiable map. Fix $\bar{x} \in U$ and let $r^*>0$ be given such that  $\overline{B_{r^*}(\bar{x})} \subset U$. Let $Y$ be a positive constant satisfying 
	\begin{equation}
	\label{eq:rpoly_Y}
	\norm{T(\bar{x})-\bar{x}}_{X} \le Y,
	\end{equation}
	and $Z:(0,r^*) \to [0,\infty)$ is a non-negative function satisfying 
	\begin{equation}
	\label{eq:rpoly_Z}
	\sup_{x\in\overline{B_r(\bar{x})}} \norm{DT(x)}_{X} \le Z(r) \qquad \text{for all} \quad r\in (0,r^*),
	\end{equation}
	where $DT(x)$ denotes the \Frech derivative of $T$ at $x \in U$ and $\norm{DT(x)}_{X}$ denotes the operator norm induced by  $\norm{\cdot }_{X}$. 
	We define the radii polynomial, $p: (0, r^*) \to \rr$, by the formula
	\begin{equation}
	p(r):=Z(r)r-r+Y.
	\label{eq:rpoly}
	\end{equation}
	If there exists $r_0\in (0,r^*)$ such that $p(r_0)<0$, then there exists a unique $x \in \overline{B_{r_0}(\bar{x})}$ so that $T(x)=x$.
\end{theorem}
\noindent
The version presented in Theorem \ref{thm:local_radii_polynomial} first appeared in \cite{JayAMSChapter} where its proof can also be found. Our proof of the \ck theorem will follow from applying Theorem \ref{thm:local_radii_polynomial} in two steps.  First, we construct a fixed point problem which amounts to defining $X,U, r^*$, and $T$ appropriately, and proving that if our construction has a solution then Equation \eqref{eq:IVP} has an analytic solution. Then, we construct $\overbar{x}$, and bounds, $Y,Z$, and prove that we can always find a positive value which makes the corresponding radii polynomial negative. 

%\subsection{Outline of the paper}
The remainder of the paper is organized as follows. In Section \ref{sec:ck_scalar}, we introduce notation and describe the construction of the fixed point problem in case $f$ is a scalar i.e.\ $n=1$. Then, we prove that our construction has a fixed point if $f$ is analytic by applying Theorem \ref{thm:local_radii_polynomial}. In Section \ref{sec:ck_vectorfield}, we generalize the construction to the vector field case. As in the scalar case, we prove that our fixed point problem always has a solution when $f$ is analytic. Finally, we prove that any fixed point of our construction implies the existence of an analytic solution for Equation \eqref{eq:IVP} which proves the \ck theorem for ODEs. 

%we introduce two applications for which an analytic solution can be constructed, even though the \ck theorem does not apply. This emphasizes the strength of our approach. In particular, the usual \ck theorem is a sufficient, but not necessary condition for analyticity of solutions. Since our approach does not make an assumption that $f$ is analytic, these examples show how it can be used to prove existence of analytic solutions even if the classical \ck theorem does not apply. 

%%%%%%%%%%%%%%%%%%%%%%%%%%%%%%%%%%%%%%%%%%%%%%%%%

\section{The fixed point problem for scalar equations}
\label{sec:ck_scalar}
In this section, we consider Equation \eqref{eq:IVP} for the case that $f$ is an arbitrary analytic scalar function. Specifically, we assume that $n = 1$ and for some $b > 0$, $f : (x_0 - b,x_0 + b) \to \mathbb{R}$ is analytic. Therefore, $f(x)$ may be written as a convergent Taylor series of the form
\[
f(x) = \sum_{k=0}^{\infty} c_k (x-x_0)^k \qquad \text{where} \quad c_k = \frac{f^{(k)}(x_0)}{k!} \quad \text{for} \quad  k \in \nn.
\]
We begin by defining some notation and reviewing necessary prerequisites from complex and functional analysis.  
\subsection{Preliminaries}
\label{sec:prelim}
We will work with the collection of real valued sequences denoted by
\begin{equation}
S :=  \setof{\setof{u_j}_{j = 0}^\infty : u_j \in \rr, \ 0 \leq j < \infty}.
\end{equation}
Let $S^\omega_\nu \subset S$ denote the collection of sequences which define analytic functions on $C^\omega(\dd_\nu)$ where $\dd_\nu = \setof{z \in \cc : \abs{z} < \nu}$ is the complex disc of radius $\nu > 0$. Though we are interested specifically in {\em real} analytic functions, we are only concerned with the property that a function converges to a power series. Thus, we do not make a distinction between a real analytic function converging say on an interval of radius $r> 0$, and its continuation to a complex analytic function converging on a complex disc of radius $r$. 

In order to apply Theorem \ref{thm:local_radii_polynomial}, we require a Banach space in which to work. With this goal in mind, we start by equipping $S$ with an appropriate norm. 
\begin{definition}
Fix a weight, $\nu > 0$ and define the space of weighted, absolutely summable sequences
\label{def:l1_nu}
\[
\ell^1_\nu := \left\{u \in S :  \sum_{j = 0}^{\infty} \nu^j \abs{u_j} < \infty\right\}.
\]
This is a normed vector space and we denote the norm of $u \in \ell^1_\nu$ by
\[
\norm{u}_{1, \nu} :=  \sum_{j = 0}^{\infty} \nu^j \abs{u_j} .
\]
\end{definition}
We note the obvious inclusions $\ell^1_\nu \subset S^\omega_\nu \subset S$ and each is strict. The following theorem provides a connection between $S^\omega_\nu$ and $\ell^1_\nu$.
\begin{proposition}
	\label{prop:ell1_vs_analytic}
	Fix $\nu > 1$ and suppose $g \in C^\omega(\dd_{\nu})$ with Taylor coefficients given by $u \in S^\omega_\nu$. Then $u \in \ell^1_{\nu'}$ for any $\nu' < \nu$. In fact, since $g^{(k)} \in C^\omega(\dd_{\nu})$ for any $k \in \nn$, it follows that the Taylor coefficients of $g^{(k)} \in \ell^1_{\nu'}$ as well. 
\end{proposition}
The proof can be found in \cite{Scheidemann2005}. Roughly speaking, Proposition \ref{prop:ell1_vs_analytic} says we can pass from analytic functions to $\ell^1_\nu$ sequences provided we ``give up some domain''. This trick is commonly used in rigorous numerical algorithms to obtain bounds on rounding and truncation errors for Taylor series. In our setting, the theorem gives us license to work with sequences in $\ell^1_\nu$ as opposed to $S^\omega_\nu$. The next proposition shows that it suffices to consider the case $\nu = 1$. 
\begin{proposition}
	\label{prop:tau_nu_equivalence}
	Suppose $V \subset \rr$ is an open subset and $f: V \to \rr$. For any $\tau, \nu> 0$, the initial value problem
	\begin{equation}
	\label{eq:nu_IVP}
	\frac{dx}{dt}  = f(x) \qquad x(0) = x_0
	\end{equation}
	has a solution with Taylor coefficients in $\ell^1_\tau$ if and only if the initial value problem
	\begin{equation}
	\label{eq:tau_IVP}
	\frac{d y}{d s}= \frac{\tau}{\nu} f(y) \qquad y(0) = x_0
	\end{equation}
	has a solution with Taylor coefficients in $\ell^1_\nu$. 
\end{proposition}

Proposition \ref{prop:tau_nu_equivalence} says that choosing $\nu$ is equivalent to rescaling time in Equation \eqref{eq:IVP}. We exploit this equivalence by making an a-priori choice for our function space. Specifically, we will work exclusively in the space $\ell_1^1$ and thus, we will omit $\nu$ from the notation for the remainder of the paper and simply write $\ell^1$ in place of $\ell_1^1$. Similarly, we let $\dd := \dd_1$ denote the complex unit disc and our discussion of analytic functions of a scalar variable will always refer to the set $C^\omega(\dd)$. The trade-off for fixing $\nu = 1$ is that we must work with a modified form of Equation \eqref{eq:IVP} given by 
\begin{equation}
\label{eq:rescaled_ivp}
\dot x = \tau f(x) \qquad x(0) = x_0
\end{equation}
where $\tau$ is a time rescaling parameter.

%\begin{theorem} 
%	\label{thm:radii_polynomial}
%Suppose that $X$ is a Banach space, $T:X \to X$ is a $Frech\acute{e}t$ differentiable mapping, and $\bar{x} \in X$. Let $Y \ge 0$ and $Z:(0,\infty) \to [0,\infty)$ a non-negative function satisfying that
%\[
%\begin{aligned}
%\norm{T(\bar{x})-\bar{x}}_X \le Y
%\end{aligned}
%\]
%and 
%\[
%\begin{aligned}
%\sup_{x\in\overline{B_r(\bar{x})}} \norm{DT(x)}_{B(X)} \le Z(r).
%\end{aligned}
%\]
%Define the radii polynomial
%\[
%\begin{aligned}
%p(r):=Z(r)r-r+Y
%\end{aligned}
%\]
%If there exists $r_0>0$ such that $p(r_0)<0$, then there exists a unique $x \in \overline{B_{r_0}(\bar x)}$ so that $T(x) = x$.
%\end{theorem}

Finally, we note that $C^\omega(\dd)$ is closed under point-wise multiplication. This gives rise to a multiplication operation on $\ell^1$ called the {\em  Cauchy product}. Specifically, the Cauchy product of $u,v \in \L$ is denoted as $u*v$ and given explicitly by the formula 
\begin{equation} 
\label{def:Cauchy_product}
(u \ast v )_n:=\sum_{k=0}^{n}u_{n-k}v_k.
\end{equation}
In fact, Merten's theorem implies that the Cauchy product makes $\ell^1$ into a Banach algebra. In particular, suppose $f,g \in C^\omega(\dd)$ are analytic functions with Taylor coefficients given by $u,v \in \L$, and let $w = u*v$. Then $w \in \L$ also and the function
\[
h(t) = \pwrs w_j t^j \qquad t \in \dd
\]
is well defined and satisfies $h(t) = f(t) g(t)$ as expected. Since $\ell^1$ is closed under products we define finite powers for Cauchy products in the obvious way by 
\[
u^k := \underbrace{u \ast u \dots \ast u}_{k \ \text{copies}}.
\] 
Evidently, it follows that if $u \in \ell^1$ then $u^k \in \L$ for any $k\in \nn$. To simplify some formulas involving powers of Cauchy products, we define 
\[
u^0 = \left(1, 0, 0, \dotsc\right)
\]
for any $u \in \ell^1$. 

\subsection{Taylor expansion of IVP solutions}
\label{sec:rescaled_taylor_expansion_ivp}
To motivate the construction of a fixed point problem, we consider the method of solving Equation \eqref{eq:rescaled_ivp} by power series expansion. We begin by considering  an ansatz for the solution to Equation \eqref{eq:rescaled_ivp} of the form
\begin{equation}
	\label{eq:scalar_soln_ansatz}
	x(t)=\sum_{j=0}^{\infty} a_jt^j \qquad a_j \in \rr. 
\end{equation}  
We want to prove that Equation \eqref{eq:scalar_soln_ansatz} defines an analytic function on some open interval containing zero by analyzing the coefficient sequence, $a(\tau) := \setof{a_j}_{j\in \nn} \in S$. Combining Proposition \ref{prop:ell1_vs_analytic} and Proposition \ref{prop:tau_nu_equivalence}, this is equivalent to proving that for some choice of $\tau$, $a(\tau) \in \ell^1$. 

For the moment, we suppose $\tau > 0$ is fixed and we suppress the dependence of $a$ on $\tau$. We formally plug Equation \eqref{eq:scalar_soln_ansatz} into Equation \eqref{eq:rescaled_ivp} to obtain 
\begin{equation}
\label{eq:scalar_match_powers}
\sum_{j=1}^{\infty} ja_j t^{j-1} = \tau f(x(t))=\tau\sum_{k=0}^{\infty} c_k \left(\sum_{j=0}^{\infty} a_j t^j - x_0 \right)^k.
\end{equation}
%By shifting the index and noting that $x(0) = a_0$ is required to satisfy the initial condition in Equation \eqref{eq:rescaled_ivp}, we obtain the explicit formula for coefficients of $a$ given by
%\begin{equation}
%\label{eq:recurrence_relation2}
%a_j = \left\{
%\begin{array}{ll}
%  a_0 & \text{if} \ j=0\\
%  \frac{\tau}{j}\sum_{k=0}^{M} c_k a^k_{j-1} & \text{if} \ j \ge 1
%\end{array}
%\right.
%\end{equation}
%
%We let $a(\tau) = \setof{a_j}_{j=0}^{\infty} \in S$ denote the sequence defined iteratively by Equation \eqref{eq:recurrence_relation2} where we emphasize that this sequence depends on $\tau$. As we observed in Section \ref{sec:introduction}, we know a-priori that Equation \eqref{eq:scalar_ivp} has a unique smooth solution so it follows that this solution must be given explicitly by Equation \eqref{eq:scalar_soln_ansatz}. In other words, the formal derivation by power series expansion is a-priori justified and there is no reason to continue using the word ``formal'' to describe the coefficient sequence. For the rest of the paper, we will reserve the variable $a$ to refer exclusively to the coefficient sequence for the function in Equation \eqref{eq:scalar_soln_ansatz} which is the exact solution to Equation \eqref{eq:rescaled_ivp}. 
%%\td{From this point forward, $a$ should only refer to the exact IVP solution}
Now, we impose $a_0 = x_0$ in order to satisfy the initial condition, and define the sequence
\[
\tilde{a} := \left(0, a_1,a_2,\dotsc \right) 
\] 
so the right hand side of Equation \eqref{eq:scalar_match_powers} has the form
\begin{equation}
	\label{eq:matched_series_rhs1}
	\tau f(x(t)) = \tau\sum_{k=0}^{\infty} c_k \left(\sum_{j=1}^{\infty} a_j t^j \right)^k = \tau\sum_{k=0}^{\infty} c_k \sum_{j=0}^{\infty} \tilde{a}_j^k t^j
\end{equation}
where the expressions of the form $\tilde{a}_j^k$ appearing in Equation \eqref{eq:matched_series_rhs1}, and throughout this work, represent the $j^{\rm th}$ term of the $k$-fold convolution. Specifically, 
\[
\tilde{a}_j^k := (\underbrace{\tilde{a} \ast \tilde{a} \dots \ast \tilde{a}}_{k \text{ copies}})_j
\]
as opposed to the $k^{\rm th}$ power of the real number, $\tilde{a}_j$. This should not lead to confusion as the latter will not appear in this paper.

Now, after matching like powers of Equation \eqref{eq:matched_series_rhs1} with the left hand side of Equation \eqref{eq:scalar_match_powers}, we obtain a recursive formula for the terms in $a$ given by
\begin{equation}
\label{eq:scalar_recursion}
a_j := 
\begin{cases}
x_0 & j = 0 \\
\frac{\tau}{j}\sum_{k=0}^{j-1} c_k \tilde{a}^k_{j-1} &j \ge 1.
\end{cases}
\end{equation}

\subsection{Constructing the fixed point problem}
\label{sec:scaled_fp_problem}
Now, we want to construct appropriate choices for $X, U$, and $T$ as in Theorem \ref{thm:local_radii_polynomial}. We start with a definition.
\begin{definition}
	\label{def:tail_subspace}
	For any $N \in \nn$ we define the {\em tail subspace} of $S$ to be
	\begin{equation}
	S\tail = \{u \in S : u_j=0 \ \text{ for }  0 \le j \le N\}.
	\end{equation}
	Similarly, we define the tail subspace of $\ell^1$ by $X = S\tail \cap \ell^1$ and we note that $X$ is a closed subspace of $\ell^1$. Hence, $X$ is a Banach space under the norm inherited from $
	\ell^1$. We will denote this norm by $\norm{\cdot}_{X}$ to emphasize when we are working in this subspace. 
\end{definition}
Now, we define a Banach space to work in by supposing that $N \in \nn$ is fixed and $S\tail, X$ denote the tail subspaces as defined in Definition \ref{def:tail_subspace}. Let $a(\tau)$ denote the sequence satisfying Equation \eqref{eq:scalar_recursion} where now we emphasize the dependence of this sequence on the choice of $\tau$ explicitly. Let $\hat{a}(\tau)$ denote the truncation of $a(\tau)$ embedded into $\ell^1$ defined explicitly by 
\begin{equation}
\label{eq:local_truncation}
\hat{a}(\tau)_j :=  
\begin{cases}
0 & j = 0, \ \text{or} \ j > N \\
a_j(\tau) & 1 \leq j \leq N.
\end{cases}
\end{equation}
Equation \eqref{eq:scalar_recursion} leads us to define the $\tau$-parameterized family of maps, $T_\tau : X \to S\tail$, by the formula
\begin{equation}
\label{eq:scalar_tail_map}
T_\tau(u)_j := 
\begin{cases}
0 &0\le j \le N\\
\frac{\tau}{j}\sum_{k=1}^{j-1} c_k \left(\hat{a}(\tau) + u\right)^k_{j-1} &j > N.
\end{cases}
\end{equation}
We will show in the next section that $a(\tau)$ is the unique fixed point of $T_\tau$. However, we ultimately want to show that $\hat{a}(\tau) \in X$ and we note that the map defined in Equation \eqref{eq:scalar_tail_map} does not necessarily map back into $X$ as required for Theorem \ref{thm:local_radii_polynomial}. As a consequence, we must first define an appropriate open subset, $U \subset X$, on which to restrict $T$. 

With this in mind, we note that since $f$ is analytic on the interval $(x_0-b, x_0 + b)$, for any constant $b_* \in (0, b)$, there exists positive real constants $C$, $C^*$ and $C^{**}$, which satisfy the bounds
\begin{align}
\label{eq:scalar_C_bound}
\sum_{k=0}^{\infty} \abs{c_k} b_*^k&< C \\
\label{eq:scalar_C*_bound}
\sum_{k=1}^{\infty} k\abs{c_k} b_*^{k-1}&< C^*\\
\label{eq:scalar_C**_bound}
\sum_{k=2}^{\infty} k(k-1)\abs{c_k} b_*^{k-2}&< C^{**}.
\end{align}
This is a simple consequence of Cauchy's integral formula combined with Proposition \ref{prop:ell1_vs_analytic}. A proof can be found in \cite{Scheidemann2005}. Next, we note that $\norm{\hat{a}(\tau)}_{1}$ is monotonically increasing as a function of $\tau$ and by a simple computation we have the limits
\[
\lim\limits_{\tau \to 0} \norm{\hat{a}(\tau)}_{1} = \hat{a}_0 = 0 \qquad \lim\limits_{\tau \to \infty} \norm{\hat{a}(\tau)}_{1} = \infty. 
\]
Hence, there exists a unique $\tau_0$ such that 
\[
\norm{\hat{a}(2\tau_0)}_{1} = b_*,
\] 
and therefore, $\norm{\hat{a}(\tau)}_{1}  < b_*$ for all $0<\tau \le \tau_0$. Define positive constants 
\begin{equation}
\label{eq:local_max_error}
	r^* := b_*-\norm{\hat{a}(\tau_0)}_{1} > 0
\end{equation}
\begin{equation}
\label{eq:local_max_tau}
	\tau^* := \min \left(\tau_0,\frac{Nr^*}{C+r^* C^*} \right)
\end{equation}
and define the open subset
\begin{equation}
\label{eq:local_openset}
 U:=\left\{u\in X : \norm{u}_{X} < \frac{1}{2}r^* \right\}.
\end{equation}
Note that the choice of $b_*$ is not unique. However, for any $b_* \in (0,b)$, this construction produces an appropriate subset $U \subset X$. 

Next, we will prove that the restriction of $T_\tau$ to $U$ satisfies the requirements of Theorem \ref{thm:local_radii_polynomial}. We start by defining some notation. 
\begin{definition}
	Let $u \in S$ be any real sequence. The {\em pointwise positive sequence} associated to $u$, denoted by $\abs{u} \in S$, is the sequence with terms defined by
	\[
	\abs{u}_j = \abs{u_j}.
	\]
\end{definition}
With this notation defined, we have the following lemma.
\begin{lemma}
	\label{lem:local_T_FD}
	Fix $N \in \nn, b_* \in (0, b)$ with corresponding constant $\tau^*$ as defined by Equation \eqref{eq:local_max_tau}, and $U \subset X$ as defined by Equation \eqref{eq:local_openset}. Suppose $\tau \in (0, \tau^*]$ is fixed, and let $\hat{a}$ denote the corresponding sequence defined in Equation \eqref{eq:local_truncation} where the dependence on $\tau$ is suppressed. Let $T$ denote the corresponding map defined by Equation \eqref{eq:scalar_tail_map}. Then 
	\begin{enumerate}[(i)]
		\item $T(U) \subset X$ 
		\item $T: U \to X$ is \Frech differentiable. 
	\end{enumerate}
\end{lemma}
\begin{proof}
	To prove $(i)$, note that $T$ maps into $S\tail$ by definition, so it suffices to show that for any $u \in U$, $T(u) \in \ell^1$. By a direct computation we have
	\[\begin{aligned}
	\sum_{j=0}^{\infty} \abs{T(u)_j }&= \sum_{j=N+1}^{\infty} \abs{\frac{\tau}{j}\sum_{k=1}^{j-1} c_k \left(\hat{a} + u\right)^k_{j-1}}\\
%	&\le \sum_{j=N+1}^{\infty} \frac{\tau}{j}\sum_{k=1}^{j-1} \abs{c_k} \abs{\left(\hat{a} + u\right)^k_{j-1}}\\
	&\le \frac{\tau}{N+1} \sum_{j=N+1}^{\infty} \sum_{k=1}^{j-1} \abs{c_k} \abs{\left(\hat{a} + u\right)^k_{j-1}}\\
	&\le \frac{\tau}{N+1} \sum_{k=1}^{\infty} \abs{c_k} \norm{\hat{a} + u}_1^k\\
	&<  \frac{\tau}{N+1} \sum_{k=1}^{\infty} \abs{c_k} b_*^k\\
	& \leq \frac{\tau C}{N+1}
	\end{aligned}
	\]
where the second to last line follows from Equation \eqref{eq:local_max_error} combined with the bound $\norm{u}_X < \frac{1}{2}r^*$, and the last line from Equation \eqref{eq:scalar_C_bound}. Hence, $T(u) \in \ell^1$ as required. 

Now, we show that $T$ is \Frech differentiable. Fix $u \in U$ and define a linear operator, $A(u): U \to X$, by its action on $h \in U$ given by the formula
\begin{equation}
\label{eq:local_DT} 
\left(A(u) h\right)_j =
\left\{
\begin{array}{ll}
0 &0\le j \le N\\
\frac{\tau}{j} \sum\limits_{k=1}^{j-1}k c_k \left(h * \left(\hat{a} + u \right)^{k-1} \right)_{j-1}  & j > N.
\end{array}
\right.
\end{equation}
The claim that $A(u)$ maps $U$ into $X$ follows from a computation similar to the proof of $(i)$ by applying Equation \eqref{eq:scalar_C*_bound}. We want to show that $A(u)$ is the  \Frech derivative of $T$ at $u \in U$. Let $h \in U$ be arbitrary such that $u+h\in U$ as well. By directly applying the formulas for $T(u)$ and $A(u)$, we have 
\tiny
\begin{align*}
		\abs{T(u+h)-T(u)-A(u)h}_j = & \abs{\frac{\tau}{j}\sum_{k=0}^{j-1} c_k\left( \left(\left(\hat{a} + u +h  \right)^k\right)_{j-1}-\left(\left(\hat{a} + u \right)^k\right)_{j-1}-k \left(h * \left(\hat{a} + u \right)^{k-1} \right)_{j-1}\right)} \\
		= & \abs{\frac{\tau}{j}\sum_{k=2}^{j-1} c_k\sum_{i=2}^{k}\frac{k(k-1)}{i(i-1)}\binom{i-2}{k-2}\left(h^i * (\hat{a} +u)^{k-i}\right)_{j-1}}.
\end{align*}
\normalsize
Now, passing to the pointwise positive sequences for $\hat a + u$ and $h$ and summing over $j \in \nn$ we obtain the estimate

\footnotesize
\begin{align*}
	\norm{T(u+h)-T(u)-A(u)h}_X & \leq \sum_{j=N+1}^\infty \frac{\tau}{j}\sum_{k=2}^{j-1} k(k-1)\abs{c_k}\sum_{i=0}^{k-2}\binom{i}{k-2}\abs{ \left(\abs{h}^{i+2} * \left(\abs{\hat{a} +  u}\right)^{k-2 - i}\right)_{j-1}}\\
	= & \sum_{j=N+1}^\infty \frac{\tau}{j}\sum_{k=2}^{j-1} k(k-1)\abs{c_k}    \left((\abs{h}^2 * \left(\abs{\hat a + u} + \abs{h} \right)^{k-2} \right)_{j-1}\\
%	\le& \frac{\tau \norm{h}^2_X}{N+1}\sum_{k=2}^{\infty} k(k-1)\abs{c_k}\norm{\abs{\hat{a}} + \abs{u} +\abs{h}}^{k-2}_1\\
		\le& \frac{\tau \norm{h}^2_X}{N+1}\sum_{k=2}^{\infty} k(k-1)\abs{c_k} \left( \norm{\hat{a}}_1 + \norm{u}_X + \norm{h}_X\right)^{k-2}\\
	<&\frac{\tau \norm{h}^2_X}{N+1}\sum_{k=2}^{\infty} k(k-1)\abs{c_k}b_*^{k-2}\\
	\le&\frac{\tau C^{**}}{N+1} \norm{h}^2_{X}.
\end{align*}
\normalsize
where the second to last line follows from Equation \eqref{eq:local_max_error} combined with the bounds $\norm{h}_X < \frac{1}{2}r^*$ and $\norm{u}_X < \frac{1}{2}r^*$, and the last line follows from Equation \eqref{eq:scalar_C**_bound}.
%More specifically speaking, since $u\in U$ and $u+h\in U$, so we have, $\norm{u}_{X} < \frac{1}{3}r^*$ and $\norm{u+h}_{X} < \frac{1}{3}r^*$, and thus, we have,
%\begin{align*}
%	\norm{h}_{X} = \norm{u+h - u}_{X} \leq \norm{u+h}_{X} + \norm{u}_{X} < \frac{2}{3}r^*
%\end{align*}
%Therefore, we can prove the accuracy from the second to last line by following,
%\begin{align*}
%	\norm{\hat{a}}_1 + \norm{u}_X + \norm{h}_X < \norm{\hat a}_{1} + \frac{1}{3}r^* + \frac{2}{3}r^* = b_*
%\end{align*}
%The last line follows from Equation \eqref{eq:scalar_C**_bound}. 
 It follows that
	\begin{equation}
	\label{def: local_def for FD}
	\lim_{\norm{h}_{X}\to 0}\frac{\norm{T(u+h)-T(u)-A(u)h}_{X}}{\norm{h}_{X}} = 0
	\end{equation}
	which proves that $T$ is \Frech differentiable. Moreover, since $0 < \tau \leq \tau^*$ was arbitrary, we have shown that $T_\tau$ is \Frech differentiable for the entire family of $\tau$-parameterized maps defined by Equation \eqref{eq:scalar_tail_map}.
\end{proof}
Lemma \ref{lem:local_T_FD} proves that $DT_\tau$ is \Frech differentiable, and moreover, its derivative is given by the formula in Equation \eqref{eq:local_DT}.
For the remainder of this work, we let $DT_\tau(u)$ denote the \Frech derivative of $T_\tau$ at $u \in U$.

%\begin{equation}
%\label{eq:local_DT} 
%\left(D(u)\cdot h\right)_j =
%\left\{
%\begin{array}{ll}
%0 &0\le j \le N\\
%\frac{\tau}{j} \sum\limits_{k=1}^{j-1}k c_k \left(h * \left(\hat{a} + u \right)^{k-1} \right)_{j-1}  & j > N
%\end{array}
%\right.
%\end{equation}

\subsection{Constructing the bounds}
To construct the bounds required for Theorem \ref{thm:local_radii_polynomial}, we begin by defining $\bar x := 0_{\ell^1} \in X$ which is the sequence of infinitely many zeroes. This choice is made independent of $N$ or $\tau$.  We are left with constructing $r_0$, $Y_\tau$ and $Z_\tau: (0, r^*) \to [0, \infty)$, such that the corresponding radii polynomial, $p_\tau(r_0) < 0$. Here the $\tau$ subscript emphasizes that these bounds depend on $\tau$.
%atisfying the appropriate uch that the following inequalities hold. 
%\begin{equation}
%\label{eq:local_Y0}
%\norm{T_\tau(\bar x)}_{X}\leq Y_\tau \\
%\end{equation}
%\begin{equation}
%\label{eq:local_Z}
%\sup \limits_{x \in \overbar{B_r(\bar x)}} \norm{DT_\tau(x)}_{X} \leq Z_\tau(r),
%\end{equation}
%\begin{equation}
%\label{eq:local_r0}
%Z_\tau(r_0)r_0 - r_0 + Y_\tau < 0
%\end{equation}
The next lemma establishes the required bounds for $Y_\tau$ and $Z_\tau$. 
\begin{lemma}
	\label{lem:local_Y0_Z0_bounds}
	Fix $N \in \nn$ and let $S\tail$ be the tail subspace of order $N$. Fix $b_* \in (0, b)$ with corresponding constants $C, C^*, r^*$ and $\tau^*$ as defined in Equations \eqref{eq:scalar_C_bound}, \eqref{eq:scalar_C*_bound}, \eqref{eq:local_max_error}, \eqref{eq:local_max_tau}, and $U \subset X = S\tail \cap \ell^1$ as defined in Equation \eqref{eq:local_openset}. For $\tau \in (0, \tau^*]$, let $\hat{a}(\tau)$ denote the truncation defined in Equation \eqref{eq:local_truncation}, and $T_\tau : U \to X$ denotes the parameterized family of maps defined in Equation \eqref{eq:scalar_tail_map}. Define the constant
\begin{equation}
	Y_\tau:=\frac{\tau C}{N+1}
\end{equation}
and the constant function, $Z_\tau : (0, r^*) \to [0, \infty)$, by the formula
\begin{equation}
	Z_\tau(r):=\frac{\tau C^* }{N+1} \qquad \text{for all} \quad r \in (0, r*).
\end{equation}
Then, the following bounds hold
\begin{equation}
\label{eq:local_Y0}
\norm{T_\tau(0)}_{X}\leq Y_\tau 
\end{equation}
\begin{equation}
\label{eq:local_Z}
\sup \limits_{u \in \overbar{B_r(0)}} \norm{DT_\tau(u)}_{X} \leq Z_\tau(r) \qquad \text{for all} \quad r \in (0, r^*).
\end{equation}
\end{lemma}

\begin{proof}
	To establish the bound for $Y_\tau$, we compute
	\[
	\begin{aligned}
	\norm{T(0)}_{X} &=\sum_{j=N+1}^{\infty} \abs{ \frac{\tau}{j}\sum_{k=0}^{j-1} c_k \hat{a}^k_{j-1}}\\
	&\le \frac{\tau}{N+1} \sum_{k=1}^{\infty}\sum_{j=N+1}^{\infty}\abs{c_k}  \abs{\hat{a}^k_{j-1}}\\
	&\le \frac{\tau}{N+1} \sum_{k=1}^{\infty} \abs{c_k} \norm{\hat{a}^k }_1\\
	&\le \frac{\tau C}{N+1}
	\end{aligned}
	\]
	which proves the bound in Equation \eqref{eq:local_Y0}. 
	
	Next, we fix $0 < r < r^*$ and $u \in \overbar{B_r(0)}$, and suppose $h \in U$ is arbitrary. Then, we have the bound
	\[
	\begin{aligned}
	\norm{DT_\tau(u)h}_{X}&=\sum_{j=N+1}^{\infty}\frac{\tau}{j}\abs{\sum_{k=1}^{\infty}k c_k(h\ast\left(\hat{a} + u \right)^{k-1})_{j-1}}\\
	&\le \frac{\tau}{N+1}\sum_{k=1}^{\infty} k \abs{c_k} \norm{h\ast\left(\hat{a} + u \right)^{k-1}}_1\\
	&\le \frac{\tau\norm{h}_{X}}{N+1} \sum_{k=1}^{\infty} k \abs{c_k} \left(\norm{\hat{a}}_1 + \norm{u}_X \right)^{k-1}.
	\end{aligned}
	\]
	Dividing through by $\norm{h}_X$, we obtain the operator norm bound
	\[
	\norm{DT_\tau(u)}_X \leq  \frac{\tau}{N+1} \sum_{k=1}^{\infty} k \abs{c_k} \left(\norm{\hat{a}}_1 + \norm{u}_X \right)^{k-1}.
	\]
	Upon taking the supremum over all $u \in \overbar{B_r(0)}$ we obtain the bound
	\[
	\sup_{u \in \overbar{B_r(0)}} \norm{DT_\tau(u)}_{X} \leq \frac{\tau}{N+1} \sum_{k=1}^{\infty} k \abs{c_k} \left(\norm{\hat{a}}_1 + r \right)^{k-1},
	\]
	and finally, we obtain a bound which holds for any $r \in (0,r^*)$ given by 
	\begin{align}
	\sup_{u \in \overbar{B_r(0)}} \norm{DT_\tau(u)}_{X} & \leq \frac{\tau}{N+1} \sum_{k=1}^{\infty} k \abs{c_k} \left(\norm{\hat{a}}_1 + r^* \right)^{k-1} \\
	& \le\frac{\tau C^*}{N+1} 
	\end{align}
	where the last line follows from Equations \eqref{eq:scalar_C*_bound} and \eqref{eq:local_max_error}. 
\end{proof}
We note that our definition of $Z_\tau$ is Lemma \ref{lem:local_Y0_Z0_bounds} is in fact a constant function with no dependence on $r$. However, the statement of Theorem \ref{thm:ck_for_analytic_scalars} allows for $Z$ to depend on $r$. In practical applications of the radii polynomial approach, bounding higher order derivatives of $DT_\tau$ yields more accurate approximations and in this case, $Z$ does indeed depend on $r$. In order to highlight the similarity between these practical applications and our proof in the present work, we will continue to consider $Z_\tau$ as a function defined on the interval $(0, r^*)$, and write $Z_\tau(r)$ despite the fact that it is constant.

We have now constructed all of the necessary ingredients for applying Theorem \ref{thm:local_radii_polynomial} which we apply to prove a precursor to the \ck theorem for the scalar case. 
\begin{theorem}[\ck precursor]
	\label{thm:ck_for_analytic_scalars}
	Suppose $V \subset \rr$ is an open subset and $f : V \to \rr$ is analytic with a Taylor expansion centered at $x_0 \in V$ given by the formula
	\[
	f(x) = \sum_{k=0}^\infty c_k (x-x_0)^k 
	\]
	which converges for $x \in (x_0 - b, x_0 + b)\subseteq V$. For any $N \in \nn$, there exists $\tau > 0$ such that the map defined by Equation \eqref{eq:scalar_tail_map} has a fixed point. 
%Then, the initial value problem
%	\begin{equation}
%	\dot x = f(x), \qquad x(0) = x_0 
%	\end{equation}
%has a unique analytic solution. 
\end{theorem}
\begin{proof}
%	By Proposition \ref{prop:tau_nu_equivalence}, the \ck theorem is equivalent to proving the existence of $\tau > 0$ such that 
%	\[
%	\dot x = \tau f(x) \qquad x(0) = x_0
%	\]
%	has a unique solution in $C^\omega(\dd)$. Fix $N \in \nn$, and
Let $S\tail$ be the tail subspace of order $N$ and let $X = S\tail \cap \ell^1$. Fix $b_* \in (0, b)$ with corresponding constants $r^*$ and $\tau^*$ as defined by Equations \eqref{eq:local_max_error}, \eqref{eq:local_max_tau}, and $U \subset X$ as defined by Equation \eqref{eq:local_openset}. Let $\hat{a}(\tau^*)$ denote the truncation defined in Equation \eqref{eq:local_truncation}, and $T_{\tau^*} : U \to X$ denotes the map defined in Equation \eqref{eq:scalar_tail_map}. Define the radii polynomial
	\[
	\begin{aligned}
	p(r) := Z_{\tau^*}(r)r - r + Y_{\tau^*} \qquad \text{for} \quad  r \in (0, r^*)
	\end{aligned}
	\]	
	where $Y_{\tau^*}$ and $Z_{\tau^*}$ are the norm bounds for $T_{\tau^*}$ and $DT_{\tau^*}$ proved in Lemma \ref{lem:local_Y0_Z0_bounds}. Applying the formulas for $Y_{\tau^*}, Z_{\tau^*}$, we obtain the bound
	\begin{align*}
			p(r) & = \frac{\tau^* C^*}{N+1}r - r + \frac{\tau^* C}{N+1} \\
%			& \leq \frac{N r^* C^*}{(N+1)(C + r^* C^*)}r + \frac{N r^* C}{(N+1)(C + r^* C^*)} - r \\
			& \leq \frac{N r^*}{(N+1)(C + r^* C^*)} \left(rC^* + C \right) - r
	\end{align*}
	for all $r \in (0,r^*)$.  \\
	\noindent
	 Define $r_0 := \frac{N}{N+1}r^* \in (0, r^*)$, and we obtain the bound
	\begin{align*}
	p(r_0) & <  \frac{N r^*}{(N+1)(C + r^* C^*)} \left(r^*C^* + C \right) - \frac{N}{N+1}r^* \\
	& = 0.
\end{align*}
By Theorem \ref{thm:local_radii_polynomial}, we conclude that $T_{\tau^*}$ has a fixed point in $U$. 
%By our assumption, a fixed point of $T_{\tau^*}$ implies that the IVP
%\[
%\dot x = \tau^* f(x), \qquad  x(0) = x_0
%\]
%has an analytic solution which completes the proof. 
%	\[
%	p_\tau(r):=\frac{\tau C^*}{N+1} r- r+ \frac{\tau C}{N+1}
%	\]
	% if $T_\tau$ has a fixed point, then we must show
	%If there exists $r_0>0$ such that $p(r_0)<0$, then there exists a unique 
	%\[
	%\tilde{u}\in\overline{B_{r_0}(0)}\subset\{u\in \ell^1\mid u_j=0\ for\ 0\le j\le N\}
	%\]
	%so that there exists an analytic solution to IVP whose coefficients are the coordinates of $a:=\hat{a}+u$. Further, since the time has been stretched by a factor $\frac{1}{\tau}$, the existent interval of this solution is $(-\tau,\tau)$.
\end{proof}
Note that Theorem \ref{thm:ck_for_analytic_scalars} implies the \ck theorem under the additional assumption that fixed points of our construction correspond to analytic solutions of Equation \eqref{eq:IVP} which we prove in the next section.

\section{The \ck theorem for analytic vector fields}
\label{sec:ck_vectorfield}
We begin by extending the construction in Section \ref{sec:ck_scalar} to the case for which $f$ is a vector field. The main technical results are already handled in the scalar case and much of the work here amounts to setting up appropriate notation so that the previous fixed point problem is meaningful. Once this is accomplished, our proof of the \ck theorem follows by first proving that fixed points of our construction imply analytic solutions of \ref{eq:IVP}, and then proving a general version of Theorem \ref{thm:ck_for_analytic_scalars} for analytic vector fields. We begin by recalling the definition of analyticity for vector fields.   
\begin{definition}
	Let $V \subset \rr^n$ be an open subset and suppose $g: V \to \rr$ is a scalar function of the $n$ variables, $\setof{x_1,\dotsc,x_n}$, which we write as components of a vector, $x \in \rr^n$.  To avoid confusion over the meaning of indices we will index the components of a vector with superscripts by writing $x = \left(x^{(1)}, \dotsc, x^{(n)}\right)$. Then, $g$ is analytic if for every
	$x = \left(x^{(1)}, \dotsc, x^{(n)}\right) \in V$, and for each $1 \leq i \leq n$, there exists an open neighborhood, $V_{x,j} \subset \rr$, containing $x^{(j)}$ such that the formula
	\[
	g_{x,j}(t) := g\left(x^{(1)}, \dotsc, x^{(j-1)}, t , x^{(j+1)}, \dotsc, x^{(n)}\right) \qquad t \in V_{x,j},
	\]
	defines an analytic function. 
	
	This definition generalizes to vector fields as follows. Suppose $g: V \to \rr^n$ is a vector field which we write as a vector of component functions, $g(x) = \left(g^{(1)}(x), \dotsc, g^{(n)}(x)\right) \in \rr^n$. Then, we define $g$ to be analytic if for each $1 \leq i \leq n$, the component function, $g^{(i)} : V \to \rr$, is analytic. 
\end{definition}
In this setting, the analog of Equation \eqref{eq:rescaled_ivp} is the initial value problem
\begin{equation}
\label{eq:vector_ivp}
	\dot x = \tau f(x)  \qquad x(0) = x_0 \in V
\end{equation}
where $V \subset \rr^n$ is an open subset, $f: V \to \rr^n$ is an analytic vector field, and $\tau > 0$ is a time rescaling parameter. The solution of Equation \eqref{eq:vector_ivp} is a function, $x : \rr \to \rr^n$, which parameterizes a trajectory of the ODE initially passing through the point $x_0$ at time $t = 0$. Our goal is to prove that if $f$ is analytic, then for each $x_0 \in V$, there exists an open interval, $J(x_0) \subset \rr$ containing $0$, such that $x: J(x_0) \to \rr^n$ defines an analytic curve.

We will construct a fixed point problem similar to the scalar case. In this version, we describe this operator at a higher level for which the construction in Section \ref{sec:ck_scalar} is a special case. Next, we introduce a Banach space to work in and define some additional notation. 
\subsection{Products of sequence spaces}
We start by generalizing the sequence spaces introduced for scalar functions in Section \ref{sec:prelim} to the vector field setting. We consider coefficient sequences  in the product
\begin{equation}
\label{eq:seq_space_product}
S^n := \underbrace{S \times S \times \dots \times S}_{n \ \text{copies}}.
\end{equation}
For arbitrary $u \in S^n$, we write $u = \left(u^{(1)}, \dotsc, u^{(n)}\right)$ with $u^{(i)} \in S$ for $1 \leq i \leq n$. If $g : \dd \to \rr^n$ is an analytic curve, then $g$ is defined by a convergent Taylor series of the form
\begin{equation}
\label{eq:analytic_curve_notation1}
	g(z) = 
	\begin{pmatrix}
	g^{(1)}(z) \\
	\vdots \\
	g^{(n)}(z)
	\end{pmatrix}
	=
	\begin{pmatrix}
	\pwrs u^{(1)}_j z^j\\
	\vdots \\
	\pwrs u^{(n)}_j z^j
	\end{pmatrix}
	\qquad u^{(i)}_j \in \rr \quad \text{for all} \quad j\in \nn, \ 1 \leq i \leq n.
\end{equation}
Hence, $g$ is naturally identified with an element, $u \in S^n$, where $u^{(i)} \in S$ is the sequence of Taylor coefficients for the analytic scalar function, $g^{(i)} : \dd \to \rr$.

Often, it is advantageous to consider an alternative description of $S^n$ in which we define elements of $S^n$ as sequences of vectors in $\rr^n$. Specifically, we have the following equivalent characterization 
\begin{equation}
\label{eq:alt_def_Sn}
S^n = \setof{ \setof{u_j}_{j = 0}^\infty : u_j \in \rr^n, \ j \in \nn}.
\end{equation}
In this case, the equivalent expression for Equation \eqref{eq:analytic_curve_notation1} can be written as 
\begin{equation}
\label{eq:analytic_curve_notation2}
	g(z) = 
	\pwrs
	u_j z^j
	\qquad u_j \in \rr^n \quad \text{for all} \quad j \in \nn.
\end{equation}
For arbitrary $u \in S^n$ we write $u^{(i)} \in S$ to express the $i^{\rm th}$ component sequence, and we write $u_j \in \rr^n$ to denote the $j^{\rm th}$ term when we consider $u$ to be an infinite sequence of real vectors. 

Following the radii polynomial approach and the constructions in Section \ref{sec:ck_scalar}, we want to work in a Banach space of absolutely summable sequences. The appropriate space for representing analytic curves would be a product of the form $\ell^1_{\nu_1} \times \ell^1_{\nu_2} \times \dots \ell^1_{\nu_n}$. By an easy generalization of Proposition \ref{prop:tau_nu_equivalence}, we can take $\nu_i = 1$ for $1 \leq i \leq n$. 
With this in mind, we define the product
\[
(\ell^1)^n := \underbrace{\ell^1 \times \ell^1 \times \dots \times \ell^1}_{n \ \text{copies}}
\]
where we note the inclusion, $(\ell^1)^n \subset S^n$. We equip $(\ell^1)^n$ with the norm defined by 
\[
\norm{u}_{\infty} := \max\setof{\norm{u^{(1)}}_1, \norm{u^{(2)}}_1 , \dotsc, \norm{u^{(n)}}_1}
\]
which makes $(\ell^1)^n$ into a Banach space.  
%\td{We either need to define $T$ on $S^n$ because even if $u \in (\ell^1)^n$ it isn't always true that $f(\hat a + u) \in (\ell^1)^n$. So the maps, $\pi_N$ and $I$ both need to be extended to a larger tail space which may include sequences that aren't summable. In fact, this is one way of thinking about why the \ck theorem is true. In this context it says that if $f$ is analytic, then whenever $u \in (\ell^1)^n$, then $f(\hat a + u) \in (\ell^1)^n$ also.}
Before continuing to the construction of the fixed point operator, we introduce notation to connect analytic functions and their Taylor coefficient sequences. 
%\begin{definition}[Taylor coefficient map]
%	\label{def:taylor_transformation}
%	Let $C^\omega(D_r(t_0),\rr^n)$ $\tlr_{t_0,r}:C^\omega(D_r(t_0),\rr^n)\to S^n$, $0<r\le 1$ that maps a bounded analytic function to the coefficients of its taylor series expansion at $t_0$. Further, we define the inverse operator of this power series transformation operator. The Inverse Power series transformation operator is given by,
%	\[
%	\label{def:inverse_taylor_transformation}
%	\tlr^{-1}_{t_0,r}(a) = \sum_{j=0}^{\infty} a_j (t-t_0)^j.
%	\]
%	where $a\in S^n$, $\abs{t-t_0}\le r$
%\end{definition}
\begin{definition}
	\label{def:taylor_transformation}
	Let $C^\omega(\dd, \rr^n)$ denote the space of parameterized curves which are analytic on $\dd$. The {\em Taylor coefficient map}, $\tlr : C^\omega(\dd, \rr^n) \to S^n$, is the linear operator which maps an analytic function to its sequence of Taylor coefficients. Specifically, $u = \tlr g \in S^n$ is the sequence defined by the formula
	\[
	u_j = 
	\begin{cases}
	g(0) & j = 0 \\
	\frac{g^{(j)}(0)}{j!} & j \geq 1.
		\end{cases}
	\]
	We define the ``inverse'' Taylor coefficient map by the formula 
	\[
	\label{def:inverse_taylor_transformation}
	\tlr^{-1} u = \pwrs u_j z^j,
	\]
	where we note that strictly speaking, $\tlr^{-1}$ is not a true inverse since $\tlr^{-1} u$ does not generally define an analytic function. Nevertheless, $\tlr^{-1} u$ is well defined as a formal power series and as we make no assumption about its convergence this notation should not present any ambiguity. 
\end{definition}
Now, we have all of the necessary ingredients to describe the construction of the fixed point operator.

\subsection{Constructing the fixed point problem}
Our first goal is to construct a fixed point problem to which we will apply Theorem \ref{thm:local_radii_polynomial}. We  start by noting that Equation \eqref{eq:vector_ivp} has a unique smooth solution, $x: J(x_0) \to \rr^n$, which follows from the same bootstrap argument as in the scalar case. Therefore, the sequence $\tlr(x) \in S^n$, is well defined.
%and we express the solution as an ansatz 
%\begin{equation}
%\label{eq:vector_soln_ansatz}
%x(t) :=\left(x^{(1)}(t),\dotsc, x^{(n)}(t) \right) = \sum_{j=0}^{\infty} a(\tau)_j t^j,\quad a(\tau)_j \in \rr^n,
%\end{equation}
%where we have written $a(\tau)$ to emphasize the dependence on the parameter in Equation \eqref{eq:vector_ivp}. 

Following the radii polynomial approach, we want to identify a fixed point problem which has a solution if and only if there exists some $\tau$ such that $a(\tau) \in (\ell^1)^n$. Next, we extend Definition \ref{def:tail_subspace} to $S^n$. 
\begin{definition}
\label{def:field_tail_subspace}
For a fixed $N \in \nn$, we define the {\em tail subspace} of order $N$ to be
\begin{equation}
\label{eq:vector_tail_subspace}
S\tail^n := \setof{u \in S^n : u^{(i)}_j = 0 \ \text{for} \ 0 \leq j \leq N, \ 1 \leq i \leq n}.
\end{equation} 
We let $X := S\tail^n \cap (\ell^1)^n$ denote the space of absolutely summable tails. Note that $X$ is a closed subspace of $(\ell^1)^n$ which makes $X$ into a Banach space under the norm inherited from $(\ell^1)^n$ and we denote this norm by $\norm{\cdot}_X$.
%Note that $X^1$ is a Banach algebra under the norm inherited from $(\ell^1)^n$ and the {\em %component-wise} product defined for $u,v \in S^n$ by 
%\[
%u \circledast v = \left(u^{(1)}*v^{(1)}, \dotsc, u^{(n)}*v^{(n)}\right) \in (\ell^1)^n. 
%\]
\end{definition}
%We note that when $n = 1$ we recover exactly the same Banach space from the proof of Theorem \ref{thm:ck_for_analytic_scalars}. 
%\td{If we never use the fact that $X$ is a closed subalgebra then we can probably omit the definition of this product. }

%Further, for the coefficients of $x(t)$, we have following expression,
%\[
%\begin{aligned}
%a&:=\left(\begin{matrix}a_1\\ a_2\\ \vdots\\ a_m\end{matrix}\right)
%:=\left(
%	\begin{matrix}
%	a_{1,0} &a_{1,1} &a{1,2} &\cdots\\
%	a_{2,0} &a_{2,1} &a{2,2} &\cdots\\
%	\vdots  &\vdots  &\vdots &\cdots\\
%	a_{m,0} &a_{m,1} &a{m,2} &\cdots\\
%	\end{matrix}\right)\\
%\tilde{a}_0&:=\left(\begin{matrix}\tilde{a}_{1,0}\\ \tilde{a}_{2,0}\\ \vdots\\ \tilde{a}_{m,0}\end{matrix}\right)
%:=\left(
%\begin{matrix}
%a_{1,0} &0 &0 &\cdots\\
%a_{2,0} &0 &0 &\cdots\\
%\vdots  &\vdots  &\vdots &\cdots\\
%a_{m,0} &0 &0 &\cdots\\
%\end{matrix}\right)\\
%\end{aligned}
%\]

%Moreover, for any $\alpha\in \mathbb{N}^m$
%\[
%\begin{aligned}
%(a_1^{(\alpha_1)})_0=(a_2^{(\alpha_2)})_0=\cdots =(a_m^{(\alpha_m)})_0=1
%\end{aligned}
%\]
Our fixed point problem will be formulated on the Banach space, $X$, given in Definition \ref{def:field_tail_subspace}. Specifically, we describe a parameterized family of maps, $T_\tau : X \to S^n\tail$, whose fixed points characterize the solutions of Equation \eqref{eq:vector_ivp}. Our construction for $T_\tau$ in the general case is decomposed as a composition of maps defined on $S^n$ which simplifies its analysis. We begin by defining a functional analytic extension of a smooth function defined on $\rr^n$, to a corresponding induced map on $S^n$. 
\begin{definition}
	Let $g$ be a formal power series in the variables $\setof{x^{(1)}, \dotsc, x^{(n)}}$ defined with multi-indices by the formula
	\[
	g(x) = \sum_{\alpha \in \nn^n} u_\alpha x^\alpha \qquad \text{where} \quad  u_\alpha \in \rr, \ x^\alpha = \prod_{i=1}^{n} \left(x^{(i)}\right)^{\alpha^{(i)}}.
	\]
	Formally, $g : \rr^n \to \rr$, defines a scalar valued function on $\rr^n$ and we note that evaluation of $g$ only requires evaluating sums and products. Hence, $g$ induces a map, $\phi_g : S^n \to S$, defined by the formula
	\[
	\phi_g(u) := \tlr \circ g(\tlr^{-1} u).
	\]
	We refer to this induced map as the {\em $S$-extension} of $g$. This generalizes to vector fields in the obvious way. If $g(x) = \left(g^{(1)}(x), \dotsc, g^{(n)}(x) \right)$ is a vector field where for $1 \leq i \leq n$, $g^{(i)}(x)$ is given by a power series, then the $S^n$-extension of $g$ denoted by $\phi_g : S^n \to S^n$, is defined by the formula
	\[
	\phi^{(i)}_g(u) = \tlr \circ g^{(i)}(\tlr^{-1} u) \qquad \text{for} \quad 1 \leq i \leq n.
	\]
\end{definition}
Next, we define two operators on $S^n$ which are important for our fixed point construction. 
\begin{definition}
	The {\em integration} map, denoted by $I: S^n \to S^n$, is the function whose action on $u \in S^n$ is defined by
	\begin{equation}
	\label{eq:integration_map}
	I(u)_j = \left\{
	\begin{array}{ll}
	0 & j = 0 \\
	\frac{1}{j}u_{j-1} & j \geq 1.
	\end{array}\right.
	\end{equation}
\end{definition}
\begin{definition}
For any $N \in \nn$, let $S\tail^n$ denote corresponding tail subspace of order $N$. We define the {\em tail projection} map, $\pi_N : S^n \to S\tail^n$, by its action on $u \in S^n$ given by the formula
\[
\pi_N(u)_j = 
\begin{cases}
0_{\rr^n} & 0 \leq j \leq N \\
u_j & j > N.
\end{cases}
%u = \left(u_0, u_1, \dotsc, u_N, u_{N+1}, \dotsc \right)  \xrightarrow{\pi_N} \left(0, 0, \dotsc, 0, u_{N+1}, u_{N+2}, \dotsc   \right).
\]
Note that the restriction of $\pi_N$ on $(\ell^1)^n$ is the induced map, $\pi_N : (\ell^1)^n \to X$. 
\end{definition}
Now, we describe the fixed point problem construction for vector fields. Let $\tilde x_0$ denote the embedding of $x_0$ into $(\ell^1)^n$ defined by 
\[
\tilde x_0 := \left(x_0, 0_{\rr^n}, 0_{\rr^n}, \dotsc \right).
\]
Suppose $\tau > 0$ and define the parameterized sequence $a(\tau) \in S^n$ by the formula
\begin{equation}
\label{eq:field_eq for recursion}
a(\tau)^{(i)}_j := 
\begin{cases}
x_0 & j = 0 \\
\frac{\tau}{j} \left(\phi_{f^{(i)}}(a(\tau) - \tilde x_0)\right)_{j-1} &j \ge 1
\end{cases}
\qquad  \text{for} \quad 1 \leq i \leq n.
\end{equation}
%	$ \tilde{a} := \left(0, a_1,a_2,\dotsc \right) $. 
Fix $N \in \nn$, and define the truncation
\begin{equation}
\label{eq:field_truncation}
\hat a := a - \tilde x_0 - \pi_N(a) \in (\ell^1)^n,
\end{equation}
and the parameterized family of maps, $T_\tau : X \to S\tail^n$, by the formula
\begin{equation}
\label{eq:field_tail_map}
T_\tau(u) = \tau \pi_N \circ I \circ \phi_f \left(\hat{a}(\tau) + u \right).
\end{equation}
Note that the construction in Section \ref{sec:ck_scalar} is a special case of this map when $n = 1$. Expressing $T_\tau$ as a composition of operators makes it easy to provide an explicit formula for $T_\tau$. However, it is no longer obvious that the Taylor coefficients of our IVP solution must be a fixed point of $T_\tau$. The next lemma proves this is the case. 
\begin{lemma}
	\label{prop:characterization_of_T} 
	Fix $N \in \nn$, let $T_\tau : X \to S^n$ be the map defined by Equation \eqref{eq:field_tail_map} and suppose that for some $\tau > 0$, $T_\tau$ has a fixed point. Then, Equation \eqref{eq:vector_ivp} has a unique solution which is analytic on the open interval $(-1,1)$. 
		%	\end{itemize}
\end{lemma}
\begin{proof}
		Let $a(\tau)$ denote the sequence defined by Equation \eqref{eq:field_eq for recursion}.  By construction, if $u$ is any fixed point of $T_\tau$, then $u + \hat{a}(\tau) + \tilde x_0$ satisfies the recursive formula in Equation \eqref{eq:field_eq for recursion}. It follows that $u = a(\tau)\tail$ since Equation \eqref{eq:field_eq for recursion} is completely determined by a choice of $\tau, x_0$. Therefore, $a(\tau)\tail \in X$ is the unique fixed point of $T_\tau$. 	
		Observe that $\tlr^{-1} (a(\tau))$ defines an analytic function on $(-1,1)$ given by the formula
		\[
		x(t) :=  \tlr^{-1} (a(\tau)) =\sum_{j=0}^{\infty} a(\tau)_j t^j.
		\]		
%		For any analytic solution $x'(t)$ to Equation \eqref{eq:vector_ivp}. We can define by the convergent series.
%		\[
%		x'(t) = \pwrs a_j' t^j \qquad t \in (-1,1).
%		\]
		Since $f$ is analytic, it has a convergent power series expansion centered at $x_0$ of the form
		\[
		f(x) = \sum_{\alpha \in \nn^n} c_\alpha (x-x_0)^\alpha.
		\]
		By composing $x$ with $\tau f$, we obtain the formula
		\begin{equation}
		\label{eq:field_recursive_right}
		\tau f(x(t)) = 
%		\sum_{j=0}^{\infty} ja'_j t^{j-1} =
		 \sum_{\alpha \in \nn^n} c_\alpha \left(\sum_{j=1}^{\infty} a(\tau)_j t^j \right)^\alpha
		\end{equation}
		where we have used the fact that $a(\tau)_0 = x_0$ by definition. 
%%		Matching like powers of $t$ leads to the equality: for all $j\in \nn$,
%		\[
%		\tlr f \tlr^{-1} (a'-\tilde x_0) = \phi_f (a'-\tilde x_0)
%		\]
		By applying $\tlr$ to the right hand side of \eqref{eq:field_recursive_right} and expressing it in terms of the $\phi$ operator we obtain the formula 
		\[
		\tlr \left(\tau f(x(t))\right)_j =  \tau (\phi_f (a(\tau)-\tilde x_0))_{j-1} \qquad \text{for all} \quad j \geq 1.
		\]
		On the other hand, we can differentiate Equation \eqref{eq:field_recursive_right} term by term to obtain the formula
		\[
		\left(\tlr \dot{x}\right)_j = j a(\tau)_j \qquad \text{for all} \quad j \geq 1.
		\]
	It follows from Equation \eqref{eq:field_eq for recursion} that
		\[
		\tlr \left(\tau f(x(t))\right) = \left(\tlr \dot{x}\right)
		\]
		proving that $x$ satisfies Equation \eqref{eq:vector_ivp}. 
%		It means $x'(t)$ and $x(t)$ holds the same recursive relation, moreover, they have the same initial value which is $x_0$. Therefore, they are the same and our proof is done.
\end{proof}
The last ingredient in our fixed point problem is to define an appropriate open subset, $U \subset X$, on which we will apply Theorem \ref{thm:local_radii_polynomial}. If $f : V \to \rr^n$ is analytic and $x_0 \in V$, then each component of $f$ can be defined by power series converging (at least) for all 
\[
x \in \left(x_0^{(1)} - b_1, x_0^{(1)}  + b_1\right) \times \dots \times \left(x_0^{(n)} - b_n, x_0^{(n)}  + b_n\right)
\] where $b_i > 0$ for $1 \leq i \leq n$. We define $b_0 := \min \setof{b_i : 1 \leq i \leq n}$, and note that for $1 \leq i \leq n$,  the component, $f^{(i)} : V \to \rr^n$, defines an analytic function. Hence, $f^{(i)}$ has a power series centered at $x_0$ of the form
\[
f^{(i)}(x) = \sum_{\alpha \in \nn} c_\alpha^{(i)} (x - x_0)^\alpha,
\]
converging at least for $x \in (-b_0, b_0)^n$. We also note the following multi-variable analog of Equations \eqref{eq:scalar_C_bound}, \eqref{eq:scalar_C*_bound}, and \eqref{eq:scalar_C**_bound}. For any $b_* < b_0$, there exist positive constants $C_i, C_i^*$ and $C_i^{**}$, possibly depending on $b_{*}$, satisfying the bounds 
\begin{align*}
	\sum_{\alpha \in \nn^n} \abs{c_\alpha^{(i)}} b_*^{\abs{\alpha} } & < C_i \\
	\sum_{\alpha \in \nn^n} \sum_{m=1}^n \alpha_{m} \abs{c_\alpha^{(i)}} b_*^{\abs{\alpha} -1}&  < C_i^*\\
	\sum_{\alpha \in \nn^n} \sum_{m_1=1}^n \sum_{m_2=1}^n \alpha_{m_1} \alpha_{m_2} \abs{c_\alpha^{(i)}} b_*^{\abs{\alpha} -2} & < C_i^{**}.
\end{align*}
%\[
%\sum_{\abs{\alpha} = 0}^{\infty}\abs{c_\alpha^{(j)}} \abs{{b_{*,0}}^\alpha} < C < \infty
%\] 
The proof follows immediately from Proposition \ref{prop:ell1_vs_analytic} and the multivariate integral Cauchy integral formula which can be found in \cite{Scheidemann2005}. We let $C, C^*$, and $C^{**}$ denote the maximum values for these constants taken over $1 \leq i \leq n$. Then, we have the bounds
\begin{align}
\label{eq:field_C_bound}
\sum_{\alpha \in \nn^n} \abs{c_\alpha^{(i)}} b_*^{\abs{\alpha}}  < C^* \\
\label{eq:field_C*_bound}
\sum_{\alpha \in \nn^n} \sum_{m=1}^n \alpha_{m} \abs{c_\alpha^{(i)}} b_*^{\abs{\alpha} -1}  < C^*\\
\label{eq:field_C**_bound}
\sum_{\alpha \in \nn^n} \sum_{m_1=1}^n \sum_{m_2=1}^n \alpha_{m_1} \alpha_{m_2} \abs{c_\alpha^{(i)}} b_*^{\abs{\alpha} -2}  < C^{**} 
\end{align}
which hold for all $1 \leq i \leq n$. We apply these bounds to define an appropriate subset, $U \subset X$, on which to restrict $T_\tau$ which is similar to the scalar case. Note that $\norm{\hat{a}(\tau)}_\infty$ is monotonically increasing as a function of $\tau$ since each component has this property. Moreover, we have the limits
\[
\lim\limits_{\tau \to 0} \norm{\hat{a}(\tau)}_\infty = 0 \qquad \lim\limits_{\tau \to \infty}  \norm{\hat{a}(\tau)}_\infty  = \infty
\]
and we define $\tau_0 > 0$ to be the unique real number satisfying $\norm{\hat{a}(2\tau_0)}_\infty = b_*$. As in the scalar case, we define the following
\begin{equation}
\label{eq:field_max_error}
r^* := b_*-\norm{\hat{a}(\tau_0)}_{\infty}
\end{equation}
\begin{equation}
\label{eq:field_max_tau}
\tau^* = \min\left(\tau_0,\frac{Nr^* }{C+r^* C^*}\right)
\end{equation}
and the open subset
\begin{equation}
\label{eq:field_openset}
U:=\left\{u\in X : \norm{u}_\infty < \frac{1}{2}r^* \right\}.
\end{equation}
This completes the construction of the fixed point problem for the vector field case. Next, we have a generalization of Lemma \ref{lem:local_T_FD} to vector fields.
%\begin{definition}[convolution between matrixes]
%	For two matrixes,  $W=(w_{i,j})_{p\times q}$ and $M=(m_{i,j})_{q\times l}$ where $p,q,l\in \nn$ and $w_{i,j}\in S$ where $1\le i\le p$, $1\le j\le q$ and $m_{i,j}\in S$ where $1\le i\le q$, $1\le j\le l$. We define the convolution $W*M$ to be a $p\times l$ matrix where 
%	\[
%	(W*M)_{i,j}:=\sum_{k=1}^{q} w_{i,k}*m_{k,j}
%	\]
%	where $1\le i\le p$ and $1\le j\le l$.
%\end{definition}
%Define the derivative of $f_j$ to be,
%\[
%D f_j:=\left(\frac{\partial f_j}{\partial x_1},\frac{\partial f_j}{\partial x_2},\cdots \frac{\partial f_j}{\partial x_n}\right)
%\]
\begin{lemma}
	\label{lem:field_T_FD}
	Fix $N \in \nn$ and $b_* \in (0, b)$, with corresponding constants $r^*$ and $\tau^*$ as defined by Equations \eqref{eq:field_max_error}, \eqref{eq:field_max_tau}, and $U \subset X$ as defined by Equation \eqref{eq:field_openset}. Let $\hat{a}(\tau)$ denote the sequence defined in Equation \eqref{eq:field_truncation}, and $T_\tau$ denotes the map defined by Equation \eqref{eq:field_tail_map}. Then, for all $\tau \in (0, \tau^*]$, the following statements hold 
	\begin{enumerate}[(i)]
		\item $T_\tau(U) \subset X$.
		\item $T_\tau: U \to X$ is \Frech differentiable. In particular, the action of $DT_\tau(u)$ on $h = \left(h^{(1)}, \dotsc, h^{(n)}\right) \in U$ is given by the formula
		\[
		\left(DT_\tau(u) h\right)^{(i)} = \sum_{m = 1}^{n} \left(\tau \pi_N \circ I \circ \phi_{\nabla f^{(i)}}(\hat{a}(\tau) + u)\right)^{(m)}*h^{(m)}
		\]
		where $\nabla f^{(i)}(x) = \left(\frac{\partial f^{(i)}}{\partial x_1},\frac{\partial f^{(i)}}{\partial x_2},\cdots, \frac{\partial f^{(i)}}{\partial x_n}\right)$ denotes the gradient vector of $f^{(i)}$. 
	\end{enumerate}
\end{lemma}
%\begin{proof}
%	Here, we recall the proof in Lemma \ref{lem:local_T_FD}. By the same method we can prove this lemma easily. Moreover, we can give the formula of $A(u):U\to X$:
%	\[
%	\left(A\left(u\right)\right)^{(j)} :=\tau \pi_N \circ I \circ \tlr Df_j \left(\tlr^{-1}\left(\hat{a}+u\right)\right)
%	\]
%	
%	In order to prove $T_\tau$ is \Frech differentiable, because we use the $\norm{\cdot}_\infty$, we only need to prove $T_\tau^{(j)}$, $j=1,2,\cdots n$ is \Frech differentiable. Where $T_\tau^{(j)}$ is defined by, for any $u\in U$
%	\[
%	T_\tau^{(j)}(u):=\tau \pi_N \circ I \circ \phi^{(j)}(\hat{a}+u).
%	\]
%	
%	where it is equal to prove for any $h\in U$, there exists a positive constant $c$ such that
%	\[
%	\small
%	\begin{aligned}
%	&\norm{T_\tau^{(j)}\left(u+h\right)-T_\tau^{(j)}\left(u\right)-\left(A\left(u\right)\right)^{(j)}*h}_\infty< c\norm{h}_\infty^2\\
%	\iff \norm{\tau \pi_N \circ I& \left(\tlr f_j \tlr^{-1}\left(\hat{a}+u+h\right)-\tlr f_j \tlr^{-1}\left(\hat{a}+u\right)-\tlr Df_j \tlr^{-1}\left(\hat{a}+u\right) *h\right)}_\infty\le c\norm{h}_\infty^2
%	\end{aligned}
%	\normalsize
%	\]
%	Moreover, we can decide the constant $c=\frac{\tau C^{**}}{N+1}$. Then the following equation satisefies,
%	\begin{equation}
%	\label{def: local_def for FD}
%	\lim_{\norm{h}_\infty\to 0}\frac{\norm{T_\tau(u+h)-T_\tau(u)-A(u)*h}_\infty}{\norm{h}_\infty} = 0
%	\end{equation}
%	Thus we finish our proof.
%\end{proof}
%Further we denote the \Frech derivative of $T_\tau$ by $DT_\tau$.
The proof is an easy generalization of the proof in Lemma \ref{lem:local_T_FD} where the bound in Equation \eqref{eq:field_C**_bound} is now applied to control all of the $2^{\rm nd}$ order (and higher) partial derivatives of $f$. We note that the formula for $DT_\tau(u)$ is nothing more than the operator obtained by applying the $S^n$-extension map to each component of the Jacobian matrix for $f$.
\subsection{Constructing the bounds}
Now, we construct the bounds required for applying Theorem \ref{thm:local_radii_polynomial}. Similar to the scalar case, we choose $\bar x = \left(0_{\rr^n},0_{\rr^n},0_{\rr^n},\dotsc\right) \in (\ell^1)^n$. The necessary bounds are provided by the following generalization of Lemma \ref{lem:local_Y0_Z0_bounds}.
\begin{lemma}
	\label{lem:field_Y0_Z0_bounds}
	Fix $N \in \nn$ and $b_* \in (0, b)$ with corresponding constants $r^*, \tau^*$ as defined by Equations \eqref{eq:field_max_error} and \eqref{eq:field_max_tau}, and $U \subset X$ as defined by Equation \eqref{eq:field_openset}. Let $\hat{a}(\tau)$ denote the truncation defined in Equation \eqref{eq:field_truncation}, and $T_\tau : U \to X$ denotes the parameterized family of maps defined in Equation \eqref{eq:field_tail_map}. For $\tau \in (0, \tau^*]$, define the constant
	\begin{equation}
	Y_\tau:=\frac{\tau C}{N+1}
	\end{equation}
	and the constant function, $Z_\tau : (0, r^*) \to [0, \infty)$, by the formula
	\begin{equation}
	Z_\tau(r):=\frac{\tau C^* }{N+1}.
	\end{equation}
	Then, the following bounds hold
	\begin{equation}
	\label{eq:field_Y0}
	\norm{T_\tau(0)}_\infty \leq Y_\tau.
	\end{equation}
	\begin{equation}
	\label{eq:field_Z}
	\sup \limits_{u \in \overbar{B_r(0)}} \norm{DT_\tau(u)}_\infty \leq Z_\tau(r) \qquad \text{for all} \quad r \in (0, r^*).
	\end{equation}
\end{lemma}
The proof is similar to the proof of Lemma \ref{lem:local_Y0_Z0_bounds} with Equations \eqref{eq:field_C_bound}, \eqref{eq:field_C*_bound} providing the necessary bounds in this case.

\subsection{The constructive proof of the \ck theorem}
%\begin{equation}
%\label{eq:field_r0}
%Z_\tau(r_0)r_0 - r_0 + Y_\tau < 0
%\end{equation}
At last, we have all ingredients necessary to give a constructive proof of the \ck theorem. 
\begin{theorem}[\ck Theorem]
	Suppose $V \subset \rr^n$ is an open subset, $f : V \to \rr^n$ is analytic, and $x_0 \in V$. Then, the initial value problem
	\begin{equation}
	\label{eq:ivp_in_ck_theorem}
	\dot x = f(x), \qquad x(0) = x_0 
	\end{equation}
	has a unique analytic solution. 
\end{theorem}
\begin{proof}
%    Let $N \geq 1$ be a given positive integer, $X$, $S\tail^n$ are the corresponding absolutely summable tails and tail subspace as defined in Definition \ref{def:field_tail_subspace} and $U$ defined in \ref{eq:field_openset} is the open subset of $X$. For $\tau = \tau^*$, let $\hat a \in (\ell^1)^n$ denote the sequence as shown in \ref{eq:field_truncation}. Since $\tau$ is defined, we can decide maps $T_{\tau^*} : U \to X$ as defined in Equation \eqref{eq:field_tail_map} and $Y_{\tau^*}, Z_{\tau^*}(r)$ defined in Lemma  \ref{lem:field_Y0_Z0_bounds}. Moreover, define the radii polynomial by 
%    \[
%    \begin{aligned}
%    p_{\tau^*}(r) := Z_{\tau^*}(r)r - r + Y_{\tau^*}, \qquad r \in (0, r^*).
%    \end{aligned}
%    \]
%    Then by Theorem \ref{thm:local_radii_polynomial}, we are done if we can prove that for some $r_0 > 0$, $p_{\tau^*}(r_0) < 0$, since this implies a fixed point for $T_{\tau^*}$ which is equivalent to $a\tail \in X$ by Propsition \ref{prop:characterization_of_T}. Further, since $0<\tau^* \le \frac{Nr^* }{C+r^* C^*}$, it is easy to verify that $r_0 = \frac{Nr^*}{N+1}$ subject to $p_{\tau^*}(r_0)<0$.
   Suppose $N \in \nn$, let $S\tail^n$ be the tail subspace of order $N$, and $X = S\tail^n \cap (\ell^1)^n$. Fix $b_* \in (0, b)$ with corresponding constants $r^*, \tau^*$ as defined by Equations \eqref{eq:field_max_error} and \eqref{eq:field_max_tau}, and $U \subset X$ as defined by Equation \eqref{eq:field_openset}. 
   
   We will consider the radii polynomial obtained from the bounds in Lemma \ref{lem:field_Y0_Z0_bounds} for the parameter value $\tau = \tau^*$. In particular, let $\hat{a} := \hat{a}(\tau^*)$ denote the truncation defined in Equation \eqref{eq:field_truncation}, $T_{\tau^*} : U \to X$ denotes the map defined in Equation \eqref{eq:field_tail_map}, and define the radii polynomial
    \[
    \begin{aligned}
    p(r) := Z_{\tau^*}(r)r - r + Y_{\tau^*} \qquad \text{for} \quad  r \in (0, r^*),
    \end{aligned}
    \]	
    where $Y_{\tau^*}$ and $Z_{\tau^*}(r)$ are the norm bounds for $T_{\tau^*}$ and $DT_{\tau^*}$ proved in Lemma \ref{lem:field_Y0_Z0_bounds}. We define $r_0 := \frac{Nr^*}{N+1} \in (0, r^*)$ and by a direct computation similar to the proof of Theorem \ref{thm:ck_for_analytic_scalars}, we have $p(r_0)<0$. It follows from Theorem \ref{thm:local_radii_polynomial} that $T_{\tau^*}$ has a unique fixed point. By Proposition \ref{prop:characterization_of_T}, this fixed point is the tail of an analytic solution to Equation \eqref{eq:vector_ivp}. By Proposition \ref{prop:tau_nu_equivalence}, this sequence is in fact a rescaled coefficient sequence for an analytic solution of Equation \eqref{eq:ivp_in_ck_theorem} which completes the proof. 
\end{proof}

\subsection{An example}
The goal of this work is not to present a practical algorithm for verifying that any particular initial value problem has an analytic solution. Nevertheless, it may be instructive to demonstrate a constructive proof for an example, especially considering that the approach is inspired by rigorous numerical algorithms which do have this exact goal in mind. 

Therefore, we conclude this paper by presenting an example of the constructive proof for a toy problem. We have intentionally chosen a rather simple example in an effort to focus on the proof itself.  Additionally, the bounds chosen to demonstrate the proof in this example are intended to make the computations easy to follow rather than minimizing the approximation error as one would probably do in practice. 

\begin{example}
	Define the function $f : \rr \to \rr$ by the formula $f(x) = x(1-x)$ and consider the scalar initial value problem 
	\begin{equation}
		\label{eq:ivp_example1}
	\dot x = \tau f(x) = \tau x(1-x), \qquad x_0 = \frac{1}{2}.
	\end{equation}
	In this example $f$ is polynomial and therefore analytic. Hence, the \ck theorem implies that Equation \eqref{eq:ivp_example1} has a unique analytic solution (in fact, the exact solution is well known to be $x(t) = (1 + \exp(-\tau t))^{-1}$). We will prove this following the constructive approach described in this paper.
	
	We begin by rewriting $f$ centered at $x_0$ as $f(x) = \frac{1}{4} - (x - \frac{1}{2})^2$. So the coefficients for $f$ are $c_0 = \frac{1}{4}$, $c_2 = -1$, and $c_j = 0$ for all $j \neq 0,2$.  Since $f$ is polynomial we have $b = \infty$ and therefore we can choose $b_*$ arbitrarily. 
	
	For this example, we let $b_* = \frac{1}{2}$ and we take $N = 5$. Applying the formula in Equation \eqref{eq:scalar_recursion} we obtain the first $N$ coefficients which are
\[
a_0(\tau) = \frac{1}{2}, \quad a_1(\tau) = \frac{\tau}{4}, \quad a_2(\tau) = 0, \quad a_3(\tau) = \frac{-\tau^3}{48}, \quad a_4(\tau) = 0.
\]
Therefore, $\hat a(\tau)$ is the sequence
\[
\hat a(\tau) = \paren{0, \frac{\tau}{4}, 0,  \frac{-\tau^3}{48}, 0, 0, \dots},
\]
which is in $\ell^1$ for all finite $\tau$. Next, we define $\tau_0$ as the solution to the equation $\norm{\hat a(2 \tau_0)}_1 = b_*$. For this example, this amounts to solving $\tau_0 + \frac{1}{3} \tau_0^3 - 1 = 0$. As expected, this equation has a unique real solution which has the exact value
\[
\tau_0 = \paren{\frac{3 + \sqrt{13}}{2}}^\frac{1}{3} - \paren{\frac{2}{3 + \sqrt{13}}}^\frac{1}{3} \approx 0.8177.
\]
Following the definition in Equation \eqref{eq:local_max_error} we find, after a bit of algebra, that $r^* = b_* - \norm{\hat a(\tau_0)}_1$ is the unique real root of the cubic polynomial $4096z^3 - 6912z^2 + 5088z - 1029$. The exact value is given by
\[
r^* = \frac{9}{16} + \frac{5}{16} \paren{\frac{\sqrt{13} - 3}{2}}^\frac{1}{3} - \frac{5}{16}\paren{\frac{2}{\sqrt{13} - 3}}^\frac{1}{3} \approx 0.3070.
\]
Next, we define $C = 1, C^* = 2$ and observe that 
\begin{align*}
	C & > \frac{1}{2} = \abs{c_0} + \abs{c_2}b_*^2 \\
	C^* & > 1 = 2 \abs{c_2} b_*,
%	C^{**} &  > 2 = 2 \abs{c_2},
\end{align*}
implying $C$ and $C^*$ satisfy the bounds required by Equations \eqref{eq:scalar_C_bound} and  \eqref{eq:scalar_C*_bound} respectively. Consequently, for this choice of $b_*, N, C$ and $C^*$, we have that 
\[
\tau_0 < \frac{Nr^*}{C + r^* C^*} \approx 0.9510,
\]
and therefore we set $\tau^* = \tau_0$ as defined in Equation \eqref{eq:local_max_tau}.

Continuing with the construction, we compute $Y_{\tau^*}$ and $Z_{\tau^*}$ according to  the formulas defined in Lemma \ref{lem:local_Y0_Z0_bounds}. For this example we obtain the bounds
\begin{align*}
	Y_{\tau^*} & = \frac{\tau^*}{6} = \frac{1}{6} \paren{\paren{\frac{3 + \sqrt{13}}{2}}^\frac{1}{3} - \paren{\frac{2}{3 + \sqrt{13}}}^\frac{1}{3} }\approx 0.1353. \\
	Z_{\tau^*}(r) & = \frac{\tau^*}{3} = \frac{1}{3} \paren{\paren{\frac{3 + \sqrt{13}}{2}}^\frac{1}{3} - \paren{\frac{2}{3 + \sqrt{13}}}^\frac{1}{3} }
	 \approx 0.2726 \qquad \text{for all} \ r \in (0, r^*).
\end{align*}
As expected, the radii polynomial, $p : (0, r^*) \to \rr$ is given by the formula
\[
p(r) = Z_{\tau^*}(r)r - r + Y_{\tau^*},
\]
which is linear in $r$. 
The conclusion of Theorem \ref{thm:ck_for_analytic_scalars} is that if $p(r)$ is negative for some $r \in (0, r^*)$, then $T_{\tau^*}$ must have a fixed point and consequently, Equation \eqref{eq:ivp_example1} has an analytic solution. As in the proof of Theorem \ref{thm:ck_for_analytic_scalars}, we choose
\[
r_0 = \frac{N}{N+1}r^* = \frac{45}{96} + \frac{25}{96} \paren{\frac{\sqrt{13} - 3}{2}}^\frac{1}{3} - \frac{25}{96}\paren{\frac{2}{\sqrt{13} - 3}}^\frac{1}{3}  \approx 0.2558,
\]
and indeed we find that $p(r_0) \approx -0.0499$ which completes the proof for this example. 
\end{example}

%\begin{example}
%The classical Lorenz system is defined by a vector field on $\rr^3$given by the following formula
%\begin{equation}
%	\label{eq:lorenz}
%	f(x) = 
%	\begin{pmatrix}
%		\sigma (x_2 - x_1) \\
%		x_1 (\rho - x_3) - x_2 \\
%		x_1 x_2 - \beta x_3
%	\end{pmatrix}
%\end{equation}
%where $\setof{\sigma, \rho, \beta}$ are real parameters. The most widely studied parameter values are given by $\sigma = 10, \rho = 28$, and $\beta = \frac{8}{3}$. 
%\end{example}

%\begin{example}
%	The autonomous double gyre is the vector field on $\rr^2$ defined by 
%	\[
%	f(x) = \pi
%	\begin{pmatrix}
%		\sin(\pi x_1) \cos (\pi x_2) \\
%		-\cos(\pi x_1) \sin (\pi x_2)
%	\end{pmatrix}
%	\]
%	Let $Df(x)$ 
%\end{example}

\section*{Acknowledgements}
The authors wish to thank Konstantin Mischaikow for helpful discussions. S.K. was partially supported by NSF grant 1839294, by NIH-1R01GM126555-01 as part of the Joint DMS/NIGMS Initiative to Support Research at the Interface of the Biological and Mathematical Science and DARPA contract HR001117S0003-SD2-FP-011. T.Z. was partially supported by the Training Program for Top Students in Mathematics from Zhejiang University.

\bibliographystyle{unsrt}
%\bibliography{/users/shane/dropbox/papers/library}
\bibliography{library}
\end{document}